\documentclass[bibalpha]{amsart}
\usepackage{verbatim}
\usepackage{enumerate}

\newcommand{\interior}[1]{\ensuremath{\mathrm{Int}}(#1)}

\def \gr {\operatorname{Graph}}

\def \cl {\operatorname{cl}}

\def \cB{\mathcal{B}}
\def \cM{\mathcal{M}}

\DeclareMathOperator{\dpr}{dp-rk}

\newcommand{\acl}{\ensuremath{\mathrm{acl}}}
\newcommand{\tto}{\rightrightarrows}

\def \cl {\mathrm{cl}}
\usepackage{enumerate}
\usepackage{braket}
\usepackage{txfonts}
\newtheorem{Th}{Theorem}[section]

\newtheorem{Cor}[Th]{Corollary}
\newtheorem{Prop}[Th]{Proposition}

\newtheorem{Lem}[Th]{Lemma}

\newtheorem*{Lem*}{Lemma}

\title{tame topology over dp-minimal structures}
\author{Pierre Simon}
\address{Univ Lyon, Universit\'e Claude Bernard Lyon 1, CNRS UMR 5208, Institut Camille Jordan, 43 blvd. du 11 novembre 1918, F-69622 Villeurbanne cedex, France}
\email{simon@math.univ-lyon1.fr}
\thanks{The first author was partially supported by ValCoMo (ANR-13-BS01-0006).}

\author{Erik Walsberg}
\address{Institute of Mathematics\\
The Hebrew University of Jerusalem\\
Givat Ram\\
Jerusalem, 91904\\
Israel}
\email{erikw@math.ucla.edu} 
\thanks{The second author was supported by the European Research Council under the European Union's Seventh Framework Programme (FP7/2007-2013) / ERC Grant agreement no.\ 291111/ MODAG}

\begin{document}
\begin{abstract}
In this paper we develop tame topology over dp-minimal structures equipped with definable uniformities satisfying certain assumptions.
Our assumptions are enough to ensure that definable sets are tame: there is a good notion of dimension on definable sets, definable functions are almost everywhere continuous, and definable sets are finite unions of graphs of definable continuous ``multi-valued functions''. 
 This generalizes known statements about weakly o-minimal, C-minimal and P-minimal theories.
\end{abstract}
\maketitle
This paper is a contribution to the study of generalizations and variations of o-minimality.
O-minimality is a model-theoretic notion of \emph{tame geometry}. 
Over an o-minimal structure definable functions are piecewise continuous and there is a well-behaved notion of dimension for definable sets. 
Conditions similar to o-minimality have been investigated, such as weak o-minimality and C-minimality, which imply analogous---though weaker---tameness properties. 
More recently, it was observed in the ordered case that a purely combinatorial condition, dp-minimality, is enough to imply such properties.
The theory of dp-minimal ordered structures can be seen as a generalization of the theory of weakly o-minimal structures, see \cite{Goodrick} and \cite{Simon:dp-min}.
The present paper continues this line of work as our results hold over dp-minimal expansions of divisible ordered abelian groups.

We use a framework which includes both dp-minimal expansions of divisible ordered abelian groups and dp-minimal expansions of valued fields.
We work with a dp-minimal structure $M$ equipped with a definable uniform structure.
We assume that $M$ does not have any isolated points and that every infinite definable subset of M has nonempty interior.
It follows from work of Simon~\cite{Simon:dp-min} that these assumptions hold for a dp-minimal expansion of a divisible ordered abelian group.
It follows from the work of Johnson~\cite{John:dpmin} that our assumptions hold for a non strongly minimal dp-minimal expansions of fields, in particular for a dp-minimal expansion of a valued field.
Our main results are:

\begin{enumerate}
\item Naive topological dimension, acl-dimension and dp-rank all agree on definable sets and are definable in families.

\item A definable function is continuous outside of a set of smaller dimension.

\item Definable sets are finite unions of graphs of continuous definable correspondences $U\rightrightarrows M^l$, $U\subseteq M^k$ an open set.

\item The dimension of the frontier of a definable set is strictly less then the dimension of the set.
\end{enumerate}
A correspondence is a continuous ``multi-valued function", this is made precise below.
The third bullet is as close as we can get to cell decomposition. 
Note that we do not say anything about definable open sets.
Cubides-Kovacsics, Darni\`ere and Leenknegt~\cite{cdl:pmin} recently showed that (2)-(4) above hold for P-minimal expansions of fields.
Dolich, Goodrick and Lippel~\cite{dgl:dpmin} showed that P-minimal structures are dp-minimal so our work yields another proof of (2)-(4) for P-minimal structures.
It follows from Proposition~\ref{frontier} below that (4) above holds for expansions of ordered groups with weakly o-minimal theory, this appears to be novel.
Eleftheriou, Hasson and Keren have recently shown \cite[Lemma 4.20]{ehk} that (4) holds for non-valuational weakly o-minimal expansions of ordered groups.
Proposition~\ref{frontier} generalizes this as non-valuational weakly o-minimal expansions of ordered groups have weakly o-minimal theory by \cite[Theorem 6.7]{MMS}.

\smallskip
We would like to thank the referee for many helpful comments.

\section{Conventions and Assumptions}
Throughout $T$ is a complete NIP theory in a multi-sorted language $L$ with a distinguished home sort and $\cM$ is an $|L^+|$-saturated model of $T$ with home sort $M$.
Throughout ``definable'' without modification means ``$\cM$-definable, possibly with parameters''.
A definable set $A$ has \textbf{dp-rank} greater than $n$ if for $0 \leq i \leq n$ there are formulas $\phi_i(x, y)$ and infinite sets $B_i \subseteq M$ such that for any $(b_0,\ldots,b_n) \in B_0 \times \ldots \times B_n$ there is an $a \in A$ such that:
$$[ \mathcal M \models \phi_i(a, y) ] \longleftrightarrow [y = b_i] \quad \text{for all } 0 \leq i \leq n, y \in B_i.$$
The theory $T$ is \textbf{dp-minimal} with respect to the home sort if $M$ has dp-rank one.
We assume throughout that $\cM$ is dp-minimal.
See Chapter 4 of~\cite{Simon:book} for more about dp-ranks.

We assume that $M$ is equipped with a definable uniform structure.
We first recall the classical notion of a uniform structure on the set $M$.
We let $\Delta \subseteq M^2$ be the set of $(x,y)$ such that $x = y$.
Given $U,V \subseteq M^2$ we declare:
$$U \circ V := \{ (x,z) \in M^2 : (\exists y \in M) (x,y) \in U, (y,z) \in V\}.$$
A \textbf{basis} for a uniform structure on $M$ is a collection $\mathcal B$ of subsets of $M^2$ satisfying the following:
\begin{enumerate}
\item the intersection of the elements of $\mathcal B$ is equal to $\Delta$;
\item if $U \in \mathcal B$ and $(x,y) \in U$ then $(y,x) \in U$;
\item for all $U,V \in \mathcal B$ there is a $W\in \mathcal B$ such that $W\subseteq U\cap V$;
\item for all $U\in \mathcal B$ there is a $V\in \mathcal B$ such that $V\circ V \subseteq U$.
\end{enumerate}
The \textbf{uniform structure} on $M$ generated by $\mathcal B$ is:
 $$\tilde{\mathcal B} := \{U \subseteq M^2 : 
(\exists V\in \mathcal B)\,V\subseteq U\}.$$
 Elements of $\tilde{\mathcal B}$ are called \textbf{entourages} and elements of $\mathcal B$ are called \textbf{basic entourages}.
Given $U \in \mathcal B$ and $x \in M$ we declare 
$$U[x] := \{y : (x,y)\in U\}.$$
We say that $U[x]$ is a \textbf{ball} with center $x$.
We put a topology on $M$ by declaring that a subset $A\subseteq M$ is open if for every $x\in A$ there is a $U\in  \mathcal B$ such that $U[x] \subseteq A$.
Assumption (1) above ensures that this topology is Hausdorff.
The collection $\{ U[x] : U \in \mathcal B\}$ forms a neighborhood basis at $x$ for each $x \in M$.
Abusing terminology, we say that $\mathcal B$ is a \textbf{definable uniform structure} if there is a formula $\varphi(x,y,\bar{z})$ such that
$$ \mathcal  B = \{ \varphi(M^2,\bar{c})\;| \; \bar{c} \in D\}$$
for some definable set $D$.
We assume throughout that $M$ is equipped with a definable uniform structure $\mathcal B$.
On each $M^k$, we put the product uniform structure, generated by $\{U_1\times \cdots \times U_k : U_i \in \cB\}$ or equivalently (because of axiom (1)), by $\{U^k : U\in \cB\}$.
Given $x = (x_1,\ldots,x_k) \in M^k$ and $U \in \cB$ we declare:
$$U[x] := \{ (y_1,\ldots,y_k) : (\forall i) (x_i,y_i) \in U \} \subseteq M^k.$$
We give the main examples of definable uniform structures.
\begin{enumerate}
\item Suppose that $\Gamma$ is an $\mathcal M$-definable ordered abelian group and $d$ is a definable $\Gamma$-valued metric on $M$.
We than take $\mathcal B$ to be the collection of sets of the form 
$$\{ (x,y) \in M^2 : d(x,y) < t \} \quad \text{for } t \in \Gamma.$$
The typical case is when $\Gamma = M$ and $d(x,y) = | x - y|$.
\item Suppose that $\Gamma$ is a definable linear order with minimal element and that $d$ is a definable $\Gamma$-valued ultrametric on $M$. Then we can put a definable uniform structure on $M$ in the same way as above.
The usual case is when $M$ is a valued field.
\item Suppose that $M$ expands a group. Let $D$ be a definable set and suppose that $\{ U_{\bar{z}}: \bar{z} \in D\}$ is a definable family of subsets of $M$ which forms a neighborhood basis at the identity for the topology on $M$ under which $M$ is a topological group. Then the sets 
$$\{ (x,y) \in M^2 : x^{-1}y \in U_{\bar{z}}\} \quad  \text{for } \bar{z} \in D$$ 
form a definable uniform structure on $M$.
\end{enumerate}
We assume that $M$ satisfies two topological conditions:
\begin{enumerate}
\item $M$ does not have any isolated points.
\item (\textbf{Inf}): every infinite definable subset of $M$ has nonempty interior.
\end{enumerate}
The first assumption rules out the trivial discrete uniformity.
The second is known for certain dp-minimal structures.
In \cite{Simon:dp-min} (\textbf{Inf}) was proven for dp-minimal expansions of divisible ordered abelian groups.
This was generalized in Proposition 3.6 of \cite{dpval} where (\textbf{Inf}) was proven under the assumption that $M$ admits a definable group structure under which $M$ is a topological group and such that for every entourage $U$ and integer $n$ there is an entourage $V$ such that $(\forall y\in V[0])(\exists x\in U[0])(n\cdot y=x)$.
It follows directly from the work of Johnson that our assumptions hold for any dp-minimal expansion of a field which is not strongly minimal:
\begin{Prop}
Let $F$ be a dp-minimal expansion of a field which is not strongly minimal.
Then $F$ admits a definable uniform structure without isolated points and every infinite definable subset of $F$ has nonempty interior with respect to this uniform structure.
\end{Prop}

\begin{proof}
It is proven in Section 4 of~\cite{John:dpmin} that $F$ admits a definable topology under which $F$ is a non-discrete topological field.
It follows that $F$ admits a definable uniform structure without isolated points.
Lemma 5.2 of~\cite{John:dpmin} shows that any infinite definable subset of $F$ has nonempty interior with respect to this topology.
\end{proof}
We finally recall some general notions.
Given sets $A,B$ and $C \subseteq A \times B$ we let
$$ C_b = \{ a \in A : (a,b) \in C\} \quad \text{for any } b \in B. $$ 
We say that family of sets $\{ A_i : i \in I \}$ is \textbf{directed} if for every $i,j \in I$ there is a $k \in I$ such that $A_i \cup A_j \subseteq A_k$.
Given a subset $A$ of a topological space we let $\cl(A)$ be the closure of $A$ and $\interior{A}$ be the interior of $A$.
The \textbf{frontier} of $A$ is $\partial(A) = \cl(A) \setminus A$.
An accumulation point of $A$ is point $p$ such that every neighborhood of $p$ contains a point in $A$ other then $p$.
The set $A$ is discrete if it has no accumulation points.
The set $A$ is \textbf{locally closed} if every $p \in A$ has a neighborhood $U$ such that $U \cap A$ is closed in $U$.
A subset of a topological space is locally closed if and only if it is the intersection of a closed set and an open set.
\begin{Lem}\label{lemma:locallyclosed}
A definable locally closed set is the intersection of a definable closed set and a definable open set.
\end{Lem}

\begin{proof}
Suppose that $A$ is locally closed.
For every $p \in A$ there is a $U \in \mathcal B$ such that $U[p] \cap A$ is closed in $U[p]$. Note that this is equivalent to $U[p]\cap A=U[p]\cap cl(A)$.
Let $V$ be the union of all $U[p]$ such that $p \in A$ and $U[p] \cap A=U[p] \cap cl(A)$.
Then $V$ is open and definable and one easily checks that $V \cap \cl(A) = A$.
\end{proof}
\smallskip
\noindent \textit{Throughout this paper $C$ is a small set of parameters and $A$ is a $C$-definable subset of $M^k$.}

\section{Dimension}
In this section we develop a theory of dimension for definable subsets of $M^k$.
We begin by noting that (\textbf{Inf}) implies that $M$ eliminates $\exists^\infty$:

\begin{Lem}\label{lemma:eliminateinfty}
If $D$ is definable and $\{ A_x : x \in D\}$ is a definable family of subsets of $M$ then there is an $n$ such that if $|A_x| > n$ then $A_x$ is infinite for all $x \in D$.
\end{Lem}

\begin{proof}
A definable subset of $M$ is discrete if and only if it is finite.
Therefore the set of $x \in D$ such that $A_x$ is finite is definable.
The lemma follows by saturation.
\end{proof}
There are several natural notions of dimension on definable subsets of $M^k$.
The \textbf{naive topological dimension} of a definable set $A$ is the maximal $l$ for which there is a coordinate projection $\pi: M^k \rightarrow M^l$ such that $\pi(A)$ has non-empty interior.
The \textbf{acl-dimension}, $\dim(\bar a/C)$, of a tuple $\bar a \in M^k$ over the base $C$ is the minimal $l$ such that there is a subtuple $\bar a'\subseteq \bar a$ of length $l$ such that $\bar a\in \acl(C\bar a')$. 
The $\acl$-dimension of $A$ is defined to be
$$\dim(A) := \max \{ \dim(\bar{a}/C) : \bar a \in A \}.$$
We can replace $C$ with any base that defines $A$, so this notion of dimension does not depend on $C$ (if say $\bar a\notin \acl(C\bar a')$ for $\bar a'$ a subtuple of $\bar a$ and $C\subseteq C_1$, then we can find $\bar a_1 \equiv_{Ca'} \bar a$ such that $\bar a_1 \notin \acl(C_1\bar a')$).
It is easy to see that $\acl$-dimension is subadditive.
If $\acl$ satisfies exchange, then by \cite[Proposition 3.2]{invdp}, $\acl$-dimension coincides with dp-rank.
In this section we prove Proposition~\ref{Prop:coin} which states that naive topological dimension, $\acl$-dimension and dp-rank coincide on definable sets.

\begin{Lem}\label{lem:fulldim}
If the naive dimension, $\acl$-dimension or dp-rank of $A$ is equal to $k$, then $A$ has non-empty interior.
\end{Lem}

\begin{proof}
It is clear from the definition of naive dimension that if the naive dimension of $A$ equals $k$ then $A$ has non-empty interior.
We show that if $\dim(A) = k$ then $A$ has non-empty interior.
Our proof also shows that if $\dpr(A) = k$ then $A$ has non-empty interior.
We only use four properties of $\acl$-dimension which hold as well for dp-rank.
We first collect these properties.
Let $D,E \subseteq M^{l + k}$ be definable and let $\pi: M^{ l +k} \rightarrow M^l$ be the projection onto the first $l$ coordinates.
Then:
\begin{enumerate}
\item $\dim(D) = 0$ if and only if $D$ is finite, and $\dim(M) = 1$;
\item $\dim(D \cup E) = \max\{ \dim(D) , \dim(E) \}$;
\item $\dim$ is subadditive:
$$ \dim(D) \leq \dim[\pi(D)] + \max\{ \dim(D_b) : b \in M^k \}. $$
\end{enumerate}
See e.g. \cite[Chapter 4]{Simon:book} for proofs that these properties hold for dp-rank.
We prove the proposition by applying induction to $k$.
If $k = 1$ then (1) and (2) above imply that $\dim(A) = 1$ if and only if $A$ is infinite, and \textbf{(Inf)} implies that $A$ is infinite if and only if $A$ has non-empty interior.
This establishes the base case.

Suppose that $k \geq 2$ and that $\dim(A) = k$.
The inductive hypothesis implies, for all $b \in M$, that $\dim(A_b) = k - 1$ if and only if $A_b$ has non-empty interior in $M^{k - 1}$.
Let $B \subseteq M\times M^{k-1}$ be the set of $(b, \bar a) \in A$ such that $\bar a \notin \interior{A_b}$.
Then $\dim(B_b) \leq k - 2$ for every $b \in M$.
Subadditivity shows that $\dim(B) \leq k -1$, so by (2) we have $\dim(A \setminus B) = k$.
It suffices to show that $A \setminus B$ has interior in $M^k$.
After replacing $A$ with $A \setminus B$ we suppose that $A_b$ is an open subset of $M^{k - 1}$ for all $b \in M$.
Let $\pi: A \rightarrow M^{k - 1}$ be the projection onto the last $(k-1)$-coordinates.
If $\pi^{-1}(\bar c)$ is finite for all $\bar c \in M^{k - 1}$ then subadditivity would imply $\dim(A) \leq k - 1$.
Therefore we fix a $\bar c \in M^{k - 1}$ such that $\pi^{-1}(\bar c)$ is infinite and let $Q \subseteq M$ be the set of $b$ such that $(b, \bar c) \in A$.
For all $b \in Q$ there is a $U \in \mathcal B$ such that $\{ b \} \times U[\bar c] \subseteq A$.
Given $U \in \mathcal B$ we let $P_U \subseteq Q$ be the set of $b$ such that $\{b\} \times U[\bar c] \subseteq A$.
If $U,V, W \in \mathcal B$ and $W \subseteq U \cap V$ then $P_U \cup P_V \subseteq P_W$.
Thus $\{ P_U : U \in \mathcal B\}$ is a directed definable family of subsets of $Q$.
It follows that for every $n$ there is a $U \in \mathcal B$ such that $|P_U| \geq n$.
As $M$ eliminates $\exists^{\infty}$ there is a $U \in \mathcal B$ such that $P_U$ is infinite.
Fix such a $U$.
As $P_U$ has non-empty interior there is an open $V \subseteq P_U$.
Then $V \times U \subseteq A$.
Thus $A$ has non-empty interior.
\end{proof}
The next lemma gives a converse to Lemma~\ref{lem:fulldim}.
\begin{Lem}\label{lem:dprank}
The following are equivalent:
\begin{enumerate}
\item $A$ has dp-rank $k$;

\item there are sequences of pairwise distinct singletons $I_l = (a^l_i : i<\omega)$ for $l<k$ such that $I_0\times \cdots \times I_{k-1} \subseteq A$;

\item there are mutually $C$-indiscernible sequences of pairwise distinct singletons $I_l = (a^l_i : i<\omega)$, $l<k$, such that $I_0\times \cdots \times I_{k-1} \subseteq A$;

\item $A$ has non-empty interior;

\item $\dim(A) = k$.
\end{enumerate}
\end{Lem}

\begin{proof}
Lemma~\ref{lem:fulldim} shows that both (1) and (5) imply (4).
If $A$ has non-empty interior then there are definable open $U_0,\ldots,U_{k - 1} \subseteq M$ such that $U_0 \times \ldots \times U_{k - 1} \subseteq A$, it easily follows that (4) implies (3).
It is obvious that (3) implies (2) and easy to see that (2) implies (5).
It remains to show that (2) implies (1).
If there are sequences as in (2), then we obtain an inp-pattern of size $k$ by considering the formulas $\phi_l (x;a^l_i) := (x=a^l_i)$.
Therefore (2) implies (1).
\end{proof}
Now we can prove:
\begin{Prop}\label{Prop:coin}
The $\acl$-dimension, naive dimension and dp-rank of $A$ coincide.
\end{Prop}
In the following proof we apply the fact that coordinate projections do not increase $\acl$-dimension or dp-rank.
\begin{proof}
We prove the proposition by showing that the following are equivalent for all $n$:
\begin{enumerate}
\item the naive dimension of $A$ is at least $n$;
\item $\dim(A) \geq n$;
\item $\dpr(A) \geq n$.
\end{enumerate}
If $\pi: M^k \rightarrow M^n$ is a coordinate projection such that $\pi(A)$ has non-empty interior then Lemma~\ref{lem:dprank} implies 
$\dim \pi(A) = \dpr \pi(A) = n,$ so $\dim(A) \geq n$ and $\dpr(A) \geq n$.
Thus (1) implies both (2) and (3).
Suppose that $\dim(A) \geq n$.
There is a coordinate projection $\pi : M^k \to M^n$ such that $\dim(\pi(A))=n$.
Lemma~\ref{lem:fulldim} implies that $\pi(A)$ has non-empty interior so the naive dimension of $A$ is at least $n$.
Thus (2) implies (1).
Suppose that $\dpr(A) \geq n$.
By \cite[Corollary 3.5]{invdp}, there is a coordinate projection $\pi: M^k \to M^n$ such that $\dpr \pi(A)=n$.
Lemma~\ref{lem:fulldim} implies that $\pi(A)$ has non-empty interior, so the naive dimension of $A$ is at least $n$.
Thus (3) implies (2).
\end{proof}
The following corollary was proven in a more general setting \cite{invdp}.
We include the easy topological proof that works in this setting.
\begin{Cor}
Let $\{D_x : x\in M^l\}$ be a definable family of subsets of $M^k$. Then for any $d\leq k$, the set of parameters $x\in M^l$ for which $\dim(D_x)=d$ is definable.
\end{Cor}
\begin{proof}
The naive topological dimension is definable in families: $\dim(D_x)\geq d$ just if there is a coordinate projection of $D_x$ to some $M^d$ with non-empty interior.
\end{proof}

We say that a definable $B \subseteq A$ is \textbf{almost all} of $A$ if $\dim(A \setminus B) < \dim(A)$.
We say that a property holds \textbf{almost everywhere} on $A$ if it holds on a definable subset of $A$ which is almost all of $A$.
If $A$ is open and $A \setminus B$ has empty interior in $A$ then if follows from Lemma~\ref{lem:fulldim} that $B$ is almost all of $A$.

\begin{Lem}\label{lem:dense}
Suppose that $A$ is open.
Suppose that $B \subseteq A$ is definable and dense in $A$.
Then the interior of $B$ is dense in $A$ and $B$ is almost all of $A$.
\end{Lem}

\begin{proof}
It suffices to show that the interior of $B$ is dense in $A$.
We fix a definable open $V \subseteq A$ and show that $B$ has non-empty interior in $V$.
We only consider the case $V = A$, the general case follows in the same way.
It thus suffices to show that $B$ has nonempty interior.
For $i \leq k$ let $V_i \subseteq M$ be non-empty open definable sets such that $V_1 \times \ldots \times V_k \subseteq A$.
 For each $i \leq k$ we fix some countably infinite $I_i \subseteq V_i$.
 Applying saturation we take $W\in \cB$ such that, for each $i$, the neighborhoods $W[a]$, $a\in I_i$ are pairwise disjoint.
Then for any choice of $\bar{a} = (a_1,\ldots,a_k) \in I_1 \times \ldots \times I_k$, there is a $\bar{y} \in B \cap W[\bar{a}]$, i.e. there is a $(y_1,\ldots, y_k) \in B$ such that $y_i \in W[a_i]$ holds for every $i$.
For $i \leq k$ let $\phi_i(x, \bar{y})$ be given by $x \in W[y_i]$ where $x$ ranges over $M$ and $\bar{y} = (y_1,\ldots,y_k)$ ranges over $B$.
For every $(a_1,\ldots,a_k) \in I_1 \times \ldots \times I_k$ there is a $\bar{y} \in B$ such that for each $i \leq k$ and $b \in I_i$, $\phi_i(b, \bar{y})$ holds if and only if $b = a_i$.
Thus the formulas $\phi_i(x, \bar{y})$ witness $\dpr(B)=k$. 
Lemma \ref{lem:dprank} shows that $B$ has non-empty interior.
\end{proof}
The following corollary will prove useful:
\begin{Cor}\label{corollary:partition}
Suppose that $A$ is open and let $A_1,\ldots, A_n$ be definable sets which cover $A$.
There is an $i \leq n$ such that $A_i$ has non-empty interior in $A$.
In fact, almost every point in $A$ is in the interior of some $A_i$.
\end{Cor}

\begin{proof}
We fix a definable open $V \subseteq U$ and show that $V$ contains a point in the interior of some $A_i$.
There is an $i \leq n$ such that $A_i$ is dense in some open subset of $V$ as otherwise the union of the $A_i$ is nowhere dense.
Lemma~\ref{lem:dense} implies that this $A_i$ has non-empty interior in $V$.
\end{proof}

\section{Correspondences and Generic Continuity}
In this section we prove Proposition~\ref{proposition:continuity} which shows that a definable function $M^k \rightarrow M^l$ is continuous almost everywhere.
We prove a stronger result which, loosely speaking, states that definable ``multi-valued functions'' are continuous almost everywhere.
We first introduce the notion of a ``multi-valued function'' that we will use.
\subsection{Correspondences}
A \textbf{correspondence} $f: E \rightrightarrows F$ consists of definable sets $E,F$ together with a definable subset $\gr(f)$ of $E \times F$ such that:
$$ 0 < | \{ y \in F : (x,y) \in \gr(f) \} | < \infty \quad \text{for all } x \in E. $$
Let $f: E \tto F$ be a correspondence.
Given $x \in E$ we let $f(x)$ be the set of $y \in F$ such that $(x,y) \in \gr(f)$. Note that saturation implies that there is a $n \in \mathbb{N}$ such that $|f(x)| \leq n$ for all $x$.
The \textbf{image} of $f$ is the coordinate projection of $\gr(f)$ onto $F$.
Given a definable $B \subseteq E$ we let $f|_B$ be the correspondence $B \tto F$ whose graph is $\gr(f) \cap [B \times F]$.
We say that $f$ is \textbf{constant} if $f(x) = f(x')$ for all $x,x' \in E$.
If $|f(x)| = m$ for every $x \in E$, then we say that $f$ is an \textbf{$m$-correspondence}.
Given correspondences $f: E \tto F$ and $g: F \tto G$ we define the composition $f \circ g: E \tto G$ to be the correspondence such that
$$ \gr(f\circ g) = \gr(f) \circ \gr(g). $$
Given $U \in \cB$ we say that $(f(x),f(x')) \in U$ if for every $y \in f(x)$ there is a $y' \in f(x')$ such that $(y,y') \in U$ and for every $y' \in f(x')$ there is a $y \in f(x)$ such that $(y,y') \in U$.
We say that $f$ is \textbf{continuous} at $x \in E$ if for every $V \in \cB$ there is a $U \in \cB$ such that $(f(x),f(x')) \in V$ whenever $(x,x') \in U$.
Note that a continuous $1$-correspondence is a continuous function.
In the remainder of this paragraph we prove several simple lemmas about correspondences which will be useful.

\begin{Lem}\label{lem:split}
Let $U \subseteq M^k$ be open and definable and let $f: U \tto M^l$ be a continuous $m$-correspondence.
Every $p \in U$ has a neighborhood $V$ such that there are definable continuous functions $g_1,\ldots, g_m : V \to M^l$ such that $\gr(g_i) \cap \gr(g_j) = \emptyset$ when $i \neq j$ and
$$ \gr(f|_V) = \gr(g_1) \cup \ldots \cup \gr(g_m). $$
\end{Lem}

\begin{proof}
Fix $p \in U$.
Let $f(p) = \{q_1,\ldots, q_m\}$.
Let $W_0 \in \mathcal B$ be such that $(q_i, q_j) \notin W_0$ for all $i,j \leq m$ such that $i \neq j$ and let $W \in \mathcal B$ be such that $W \circ W \subseteq W_0$.
Let $V$ be an open neighborhood of $p$ such that $(f(p), f(p')) \in W$ for all $p' \in V$.
Fix $p' \in V$ and let $f(p') = \{q'_1,\ldots, q'_m\}$.
For each $i \leq m$ there is a $j \leq m$ such that $(q_i, q'_j) \in W$.
As the balls $W[q_i]$ are pairwise disjoint we see that for each $i \leq m$ there is a unique $j \leq m$ such that $(q_i, q'_j) \in W$.
We have shown that for every $p' \in V$ and $q \in f(p)$ there is a unique $q' \in f(p')$ such that $(q,q') \in W$.
For $i \leq m$ we let $g_i: V \rightarrow M^l$ be the definable function such that $g_i(p') \in W[q_i]$ and $g_i(p') \in f(p')$ for every $p' \in V$. Continuity of the $g_i$'s follows easily from the continuity of $f$.
It is clear that the graphs of the $g_i$ are pairwise disjoint.
\end{proof}

\begin{Lem}\label{lem:split1}
Let $U \subseteq M^k$ be open and definable and let $f: U \tto M^l$ be a continuous correspondence.
Almost every $p \in U$ has a neighborhood $V$ such that there are definable continuous functions $g_1,\ldots, g_m : V \to M^l$ such that $\gr(g_i) \cap \gr(g_j) = \emptyset$ when $i \neq j$ and
$$ \gr(f|_V) = \gr(g_1) \cup \ldots \cup \gr(g_m). $$
\end{Lem}

\begin{proof}
Let $m$ be such that $|f(p)| \leq m$ for all $p \in U$.
For each $i \leq m$ let $A_i \subseteq U$ be the set of $p$ such that
$ |f(p)| = i.$
By Corollary~\ref{corollary:partition} almost every element of $U$ is contained in the interior of some $A_i$.
An application Lemma~\ref{lem:split} shows that the conclusion of the lemma holds for any element of the interior of some $A_i$.
\end{proof}
The next lemma is a straightforward generalization of a familiar fact about graphs of continuous functions.
We leave the proof to the reader.
\begin{Lem}\label{lem:locclos}
Let $f: A \tto M^l$ be a continuous correspondence.
Then $\gr(f)$ is a closed subset of $A\times M^l$.
If $A$ is open then $\gr(f)$ is a locally closed subset of  $M^k \times M^l$.
\end{Lem}
The following lemma is well-known for continuous functions.
Lemma~\ref{lem:split} reduces Lemma~\ref{lem:localhomeo} to the case of a continuous function $f$.
We again leave the details to the reader.
\begin{Lem}\label{lem:localhomeo}
Suppose that $A$ is open and let $f: A \tto M$ be a continuous $m$-correspondence.
Let $\pi: A \times M \rightarrow A$ be the coordinate projection.
Then every $p \in \gr(f)$ has a neighborhood $V \subseteq \gr(f)$ such that $\pi(V)$ is open and the restriction of $\pi$ to $V$ is a homeomorphism onto its image.
\end{Lem}

\subsection{Generic Continuity}
In this section we prove Proposition~\ref{proposition:continuity} which states that a definable correspondence $M^k \tto M^l$ is continuous almost everywhere.
We first prove two lemmas which we use in the proof of Proposition~\ref{proposition:continuity} and in several other places.

\begin{Lem}\label{lem:poset}
Let $\mathcal C = \{ C_x : x \in M^l \}$ be a directed definable family of subsets of $M^k$.
If 
$$\bigcup_{ x \in M^l } C_x$$
has non-empty interior then there is an element of $\mathcal C$ with non-empty interior.
\end{Lem}

\begin{proof}
Suppose that the union of $\mathcal C$ has non-empty interior.
We show that there is a $k$-dimensional element of $\mathcal C$.
For $1 \leq i \leq k$ let $U_k$ be open definable subsets of $M$ such that
$$ U_1 \times \ldots \times U_k \subseteq \bigcup_{ x \in M^l } C_x.$$
For each $1 \leq i \leq k$ let $I_i \subseteq U_i$ be a countable set.
Let $I = I_1 \times \ldots \times I_k$.
As $\mathcal C$ is directed, for every finite $J \subseteq I$ there is a $y \in M^l$ such that $J \subseteq C_y$.
Saturation gives a $y \in M^l$ such that $I \subseteq C_y$.
Lemma \ref{lem:dprank} implies that this $C_y$ has non-empty interior in $M^k$.
\end{proof}
(\textbf{Inf}) implies that there are no infinite definable discrete subsets of $M$.
A straightforward inductive argument extends this to any $M^k$:
\begin{Lem}\label{lem:nodiscrete}
There is no infinite definable discrete subset of $M^k$.
\end{Lem}

\begin{proof}
We apply induction to $k$.
The base case follows from (\textbf{Inf}).
We fix $k \geqslant 2$ and suppose towards a contradiction that $D \subseteq M^k$ is definable, infinite and discrete.
For all $x \in D$ there is a $U \in \cB$ such that $U[x] \cap D = \{ x\}$.
Applying saturation fix a $U \in \cB$ such that $U[x] \cap D = \{ x \}$ holds for infinitely many $x \in D$.
After replacing $D$ with the set of such $x$ if necessary we suppose that if $x,y \in D$ and $x \neq y$ then $(x,y) \notin U$.
Let $\pi_1: M^k \to M^{ k - 1 }$ be the projection onto the first $k - 1$ coordinates and let $\pi_2: M^k \to M$ be the projection onto the last coordinate.
We first suppose that $\pi_1(D)$ is finite.
This implies that there is a $d \in \pi_1(D)$ such that $\pi_1^{-1}(d) \cap D$ is infinite.
Then $\pi_2[\pi_1^{-1}(d) \cap D]$ is infinite and discrete.
This contradicts the base case so we may assume that $\pi_1(D)$ is infinite.
Applying the inductive assumption we fix an accumulation point $w$ of $\pi_1(D)$.
Let $U' \in \mathcal{B}$ be such that $U' \circ U' \subseteq U$.
We declare $W = U'[w] \times M$ and $D' = D \cap W$.
Note that $D'$ is infinite.
If $x, y \in D'$ then $(\pi_1(x), w) \in U'$ and $(\pi_1(y), w) \in U'$ so $(\pi_1(x), \pi_1(y)) \in U$.
If $x,y \in D'$ and $(\pi_2(x), \pi_2(y)) \in U$ then as $(\pi_1(x), \pi_1(y)) \in U$ we would also have $(x,y) \in U$, which implies $x = y$.
Thus if $x,y \in D'$ and $x \neq y$ then $(\pi_2(x), \pi_2(y)) \notin U$.
This implies that $\pi_2(D')$ is discrete and therefore finite.
As $D'$ is infinite there is a $d \in \pi_2(D')$ such that $\pi_2^{-1}(d) \cap D'$ is infinite.
Then $\pi_1[\pi_2^{-1}(d) \cap D']$ is infinite and discrete.
This contradicts the inductive assumption.
\end{proof}

\begin{Prop}\label{proposition:continuity}
Let $V \subseteq M^k$ be a definable open set.
Every correspondence $V \rightrightarrows M^l$ is continuous on an open dense subset of $V$, and thus is continuous almost everywhere on $V$.
\end{Prop}

\begin{proof}
As $M^l$ is equipped with the product topology it suffices to show that every correspondence $f: V \tto M$ is continuous on an open dense set.
By Lemma \ref{lem:dense} it suffices to show that the set of points of continuity of $f: V \tto M$ is dense.
It is therefore enough to fix an open $V' \subseteq V$ and show that $f$ is continuous on some point in $V'$.
To simplify notation we assume $V' = V$, this does not result in any loss of generality.

We first treat the case $k = 1$.
We suppose towards a contradiction that $f$ is discontinuous at every point in $V$.
Let $n$ be such that $|f(p)| \leq n$ for all $p \in V$.
For every $i \leq n$ we let $A_i$ be the set of $p \in V$ such that $|f(p)| = i$.
Applying Corollary~\ref{corollary:partition} fix $i \leq n$ such that $A_i$ has non-empty interior in $V$.
After replacing $V$ with a smaller definable open set if necessary we suppose that $V \subseteq A_i$.
Let $B \subseteq \cB \times V$ be the set of $(W,p)$ such that for all $W' \in \cB$ there is a $q \in W'[p]$ such that $( f(p), f(q) ) \notin W$.
For every $p \in V$ there is a $W \in \cB$ such that $(W,p) \in B$.
As the family $\{ B_W : W \in \cB\}$ is directed we apply Lemma~\ref{lem:poset} and fix a $W \in \cB$ such that $B_W$ has non-empty interior in $V$.
After replacing $V$ with a smaller definable open set if necessary we suppose that $V \subseteq B_W$.
For every $p \in V$ there are $q \in V$ arbitrarily close to $p$ such that $( f(p), f(q) ) \notin W$.
Fix $U \in \cB$ such that $U \circ U \subseteq W$.
Let $D \subseteq V \times M^i$ be the set of $(p,\bar y)$ such that $\bar y = (y_1,\ldots, y_i)$ and $f(p) = \{ y_1, \ldots, y_i \}$.
Let $\pi: D \to V$ be the coordinate projection.
As $D$ is infinite an application of Lemma~\ref{lem:nodiscrete} gives an accumulation point $(p, \bar y) \in D$.
Thus $U[(p, \bar y)] \cap D$ is infinite, so $\pi( U[(p, \bar y)] \cap D)$ is also infinite and thus has non-empty interior in $V$.
Let $V'$ be a definable open subset of $\pi(U[(p,\bar y)] \cap D)$.
Note that if $x \in V'$ then $(f(x), f(p)) \in U$.
Fix $q \in V'$.
For all $r \in V'$ we have $(f(q), f(p)) \in U$ and $(f(r), f(p)) \in U$ so therefore $(f(q), f(r)) \in W$.
This is a contradiction as there are $r$ arbitrarily close to $q$ satisfying $(f(q),f(r)) \notin W$.
Thus $f$ must be a continuous at some point in $V$.

We now apply induction to $k \geqslant 2$.
We again suppose towards a contradiction that $f$ is discontinuous at every point in $V$.
For every $p \in V$ there is a $W \in \cB$ such that there exist $q \in V$ arbitrarily close to $p$ satisfying $( f(p), f(q) ) \notin W$.
Arguing as in the case $k = 1$ we may suppose that $W \in \cB$ is such that for all $p \in V$ there are $q \in V$ arbitrarily close to $p$ satisfying $( f(p), f(q) ) \notin W$.
After replacing $V$ with a smaller definable open set if necessary we suppose that $V = V_0 \times V_1$ for definable open $V_0 \subseteq M$ and $V_1 \subseteq M^{k - 1}$.
Given $\bar y \in V_1$ we let $f_{ \bar y}: V_0 \tto M$ be the correspondence given by $f_{ \bar y}(t) = f(t, \bar y)$. 
Then for all $\bar y \in V_1$ the correspondence $f_{\bar y}$ is continuous away from finitely many points of $V_0$.
It follows by subadditivity that the set of $(t, \bar y) \in V_0 \times V_1$ such that $f_{\bar y}$ is discontinuous at $t$ has dimension at most $k - 1$ and is therefore nowhere dense.
After replacing $V_0$ and $V_1$ with smaller definable open sets if necessary we suppose that $f_{\bar y} : V_0 \tto M$ is continuous for all $\bar y \in V_1$.
Let $U \in \cB$ be such that $U \circ U \subseteq W$.
For $O \in \cB$ let $B_O \subseteq V$ be the set of $(t, \bar y)$ such that if $t' \in O[t]$, then $(f_{\bar y}(t), f_{\bar y}(t') ) \in U$.
For every $(t,\bar y) \in V$ there is an $O \in \cB$ such that $(t,\bar y) \in B_O$.
The family $\{ B_O : O \in \cB\}$ is directed so applying Lemma~\ref{lem:poset} we fix an $O \in \cB$ such that $B_O$ has non-empty interior in $V_0 \times V_1$.
After replacing $V_0$ and $V_1$ with smaller open sets if necessary we suppose that $V_0 \times V_1 \subseteq B_O$ and $V_0 \times V_0 \subseteq O$.
Thus if $\bar y \in V_1$ and $t,t' \in V_0$ then $( f(t, \bar y), f(t', \bar y) ) \in U$.
Fix $t \in V_0$ and let $f^t : V_1 \tto M$ be given by $f^t(\bar y) = f(t, \bar y)$.
Applying the inductive hypothesis we fix a $\bar z \in V_1$ at which $f^t$ is continuous.
After replacing $V_1$ with a smaller open set if necessary we may suppose that $( f^t(\bar y ) , f^t(\bar z) ) \in U$ holds for all $\bar y \in V_1$.
Suppose that $(s, \bar y) \in V_0 \times V_1$.
Then
$$ ( f(t, \bar y), f(t, \bar z) ) \in U \quad \text{and} \quad ( f(t, \bar y), f(s, \bar y) ) \in U. $$
As $U \circ U \subseteq W$ we conclude that:
$$ ( f(t, \bar z), f(s, \bar y) ) \in W \quad \text{for all } (s, \bar y) \in V_0 \times V_1. $$
This gives a contradiction as there are $(s, \bar y)$ arbitrarily close to $(t, \bar z)$ such that $( f(t, \bar z), f(s, \bar y) ) \notin W$.
\end{proof}
Definable closure will not in general agree with algebraic closure, so it should not in general be the case that the graph of a continuous correspondence is a finite union of graphs of definable functions.
Corollary~\ref{corollary:rationalclos} allows us to make up for this in some circumstances.
Corollary~\ref{corollary:rationalclos} is a direct consequence of Lemma~\ref{lem:split1} and Proposition~\ref{proposition:continuity}.

\begin{Cor}\label{corollary:rationalclos}
Let $U \subseteq M^k$ be open and definable and let $f: U \tto M^l$ be a correspondence.
Almost every $p \in U$ has a neighborhood $V$ such that there are continuous definable functions $g_1,\ldots, g_m : V \to M^l$ such that $\gr(g_i) \cap \gr(g_j) = \emptyset$ when $i \neq j$ and
$$ \gr(f|_V) = \gr(g_1) \cup \ldots \cup \gr(g_m). $$
\end{Cor}

\section{A Decomposition}
We now show that every definable set is a finite union of graphs of correspondences.
A more complicated argument can be used to show that every definable set is a finite disjoint union of graphs of correspondences.
We do not prove this as the weaker result suffices for our purposes. As before we let $A\subseteq M^k$ be some $C$-definable subset.

\begin{Prop}\label{prp:weakcell}
There are $C$-definable sets $A_1, \ldots, A_n \subseteq A$ which cover $A$ such that each $A_i$ is, up to permutation of coordinates, the graph of a $C$-definable continuous $m$-correspondence $f: U_i \tto M^{k-d}$, where $U_i\subseteq M^{d}$ is a $C$-definable open set and $0 \leq d \leq k$.
\end{Prop}
If $d = 0$ then we identify the graph of $f: M^0 \rightarrow M^k$ with a finite subset of $M^k$.
If $d = k$ then we identify the graph of $f: U \tto M^0$ with $U$.
In this way we regard any open definable subset of $M^k$ and any finite subset of $M^k$ as the graph of a correspondence.

\begin{proof}
By saturation it suffices to prove the following: for any $\bar{a} \in A$ there is a $C$-definable set $A_0$ which is, up to a permutation of coordinates, the graph of a $C$-definable continuous $m$-correspondence $U \tto M^{k - d}$ for some $C$-definable open $U \subseteq M^d$, and satisfies $\bar{a} \in A_0 \subseteq A$.
Fix $\bar a = (a_1,\ldots,a_k) \in A$. 
Let $d=\dim(\bar a|C)$. 
By definition of dimension, up to a permutation of variables, we have $(a_{d+1},\ldots,a_k)\in \acl(Ca_1,\ldots,a_{d})$.
It follows that there is a $C$-definable set $B \subseteq M^d$ and a $C$-definable correspondence $f: B \tto M^{k - d}$ such that $\bar{a} \in \gr(f)$.
After intersecting $\gr(f)$ with $A$ and replacing $B$ with a smaller $C$-definable set if necessary we may assume that $\gr(f) \subseteq A$.
Lemma~\ref{lem:dprank} shows that $\dim(B \setminus \interior{B}) < d$ so as $(a_1,\ldots,a_d) \in B$ and $\dim(a_1,\ldots,a_d|C) = d$ we have $(a_1,\ldots,a_d) \in \interior{B}$.
Let $N$ be such that $|f(x)| \leq N$ for all $x \in B$.
For each $1 \leq i \leq N$ let $E_i \subseteq \interior{B}$ be the set of $x$ such that $|f(x)| = i$.
Corollary~\ref{corollary:partition} shows that
$$ \dim[\interior{B} \setminus (\interior{E_1} \cup \ldots \cup \interior{E_N})] < d $$
so $(a_1,\ldots,a_d) \in \interior{E_m}$ for some $1 \leqslant m \leqslant N$.
Fix such an $m$.
Let $U \subseteq \interior{E_m}$ be the set of points that have a neighborhood on which $f$ is continuous.
Proposition~\ref{proposition:continuity} shows that $U$ is almost all of $\interior{E_m}$ so $(a_1,\ldots,a_d) \in U$.
The restriction of $f$ to $U$ is a continuous $m$-correspondence.
We take $A_0 = \gr(f|_{U})$.
\end{proof}
\noindent From Proposition~\ref{prp:weakcell} and Lemma~\ref{lem:locclos} we immediately have:
\begin{Cor}\label{cor:constructible}
Every definable subset of $M^k$ is a finite union of locally closed sets.
Every definable subset of $M^k$ is a boolean combination of definable open sets.
\end{Cor}
\noindent We now show that dimension of the frontier of $A$ is strictly less then the dimension of $A$.
\begin{Prop}\label{frontier}
$\dim \partial(A) < \dim A $.
\end{Prop}
\begin{proof}
If $A = A_1 \cup \ldots \cup A_n$ then
$ \partial (A) \subseteq \partial(A_1) \cup \ldots \cup \partial(A_n). $
Therefore if $A_1,\ldots A_n \subseteq A$ are definable sets which cover $A_i$ and $\dim(A_i) < \dim \partial(A_i)$ holds for every $i$ then $\dim(A) < \dim \partial(A)$.
Applying Corollary~\ref{cor:constructible} we may assume that $A$  is locally closed.
We let $\dim \partial(A) = l$.
Let $\pi: M^k \to M^l$ be the projection onto the first $l$ coordinates.
After permuting coordinates if necessary we assume that $\pi[\partial(A)]$ is $l$-dimensional.
By Lemma~\ref{lem:dprank} there are sequences $J_m = (a^m_i : i < \omega)$ for $1 \leq m \leq l$ such that 
$$ J_1 \times \ldots \times J_l \subseteq \pi[\partial(A)].$$
Given $\bar{r} = (r_1,\ldots,r_l) \in \omega^l$ we let $a_{\bar{r}} = (a^1_{r_1},\ldots,a^l_{r_l})$.
Applying saturation we let $W_0 \in \mathcal B$ be such that
$$ W_0[a^m_i] \cap W_0[a^m_j] = \emptyset \text{ for any } 1 \leqslant m \leqslant l \text{ and distinct } i,j < \omega. $$
For every $\bar{r} \in \omega^l$ we pick an $x^0_{\bar{r}}$ in $A$ such that $W_0[x^0_{\bar{r}}]$ intersects $\partial(A) \cap \pi^{-1}(a_{\bar{r}})$.
As $A$ is locally closed, for each $x^0_{\bar{r}}$ there is a $W \in \mathcal B$ such that $W[x^0_{\bar{r}}]$ is disjoint from $\partial(A)$.
Applying saturation pick an entourage $W_1$ contained in $W_0$ such that $$ W_1[x^0_{\bar{r}}] \cap \partial(A) = \emptyset \text{ for all } \bar{r} \in \omega^l.$$
 Pick points $x^1_{\bar{r}}$ as before with $W_1$ replacing $W_0$ and iterate.
In the end we obtain a nested sequence of entourages $(W_n : n < \omega)$ and points $\{x^n_{\bar{r}} \in A : (\bar{r},n) \in \omega^{l +1}\}$ such that $W_n[x^n_{\bar{r}}]$ intersects $\partial(A) \cap \pi^{-1}(a_{\bar{r}})$ and $W_{n+1}[x^n_{\bar{r}}]$ is disjoint from $\partial(A)$ for all $(n,\bar{r})$.
We let $\psi$ be a formula such that 
$$\mathcal{B} = \{ \psi(M^2, \bar{b}) : \bar{b} \in M^q\}.$$
For each $n$ we let $\bar{b}_n \in M^q$ be such that
$$ \psi(M^2, \bar{b}_n) = W_n. $$
Given variables $\bar{x} = (x_1,\ldots,x_k)$ we define formulas: 
$$\phi_{m}(\bar{x}, a^m_i) := x_m \in W_0[a^m_i] \quad \text{ for } 1 \leqslant m \leqslant l, i < \omega.$$
and 
$$\phi_{l+1}(\bar{x}, \bar{b}_i, \bar{b}_{i + 1}) := [\partial(A) \cap W_i[\bar{x}] \neq \emptyset] \wedge [\partial(A) \cap W_{i+1}[\bar{x}] = \emptyset] \text{ for } i < \omega.$$
This yields an ict-pattern of depth $l + 1$ based on $A$.
Thus $\dim(A) \geq l + 1$.
\end{proof}
\noindent Let $B \subseteq A$.
The relative interior of $B$ in $A$ is the set of $p \in B$ for which there is an open $U \subseteq M^k$ such that $p \in U$ and $U \cap A \subseteq B$.
\begin{Cor}\label{corollary:int1}
Suppose that $B \subseteq A$ is definable and $\dim(B) = \dim(A)$.
Then the relative interior of $B$ in $A$ is almost all of $B$.
\end{Cor}

\begin{proof}
Let $I$ be the relative interior of $B$ in $A$.
Then $B \setminus I \subseteq \partial(A \setminus B)$.
Therefore:
$$ \dim (B \setminus I) \leqslant \dim \partial(A \setminus B) < \dim(A \setminus B) \leq \dim(A) = \dim(B) .$$
So $I$ is almost all of $B$.
\end{proof}

\begin{Cor}\label{corollary:int2}
Let $B_1, \ldots, B_m$ be definable subsets of $A$ which cover $A$.
Then almost every element of $A$ is contained in the relative interior of $B_i$ in $A$ for some $i \leq m$.
\end{Cor}

\begin{proof}
After permuting the $B_i$ if necessary we may suppose that $n < m$ is such that $\dim(B_i) < \dim(A)$ when $i < n $ and $\dim(B_i) = \dim(A)$ when $i \geq n$.
Then $B_n \cup \ldots \cup B_m$ is almost all of $A$.
Let $I_i$ be the relative interior of $B_i$ in $A$ for each $i \leqslant m$.
By Corollary~\ref{corollary:int1} $I_i$ is almost all of $B_i$ for every $i \geq n$.
It follows that $I_n \cup \ldots \cup I_m$ is almost all of $A$.
\end{proof}
We are mainly interested in the following proposition in the case when $M$ admits a definable group operation which is compatible with the definable uniform structure.
Then $M^k$ is also a group and is hence topologically homogeneous.
In this case we view the following proposition as stating that almost every point in $A$ is ``topologically non-singular''.

\begin{Prop}\label{prp:manifold}
Let $\dim(A) = d$.
Almost every $p \in A$ has a neighborhood $V \subseteq A$ for which there is a coordinate projection $\pi: M^k \rightarrow M^d$ such that $\pi(V)$ is open and the restriction of $\pi$ to $V$ is a homeomorphism onto its image.
\end{Prop}

\begin{proof}
Let $A_1, \ldots, A_m \subseteq A$ be definable sets which cover $A$ such that each $A_i$ is, up to permutation of coordinates, the graph of a definable continuous $m_0$-correspondence $f: U_i \tto M^{k-d}$, where $U_i\subseteq M^{d}$ is a definable open set and $0 \leq d \leq k$.
We suppose that $n \leq m$ is such that $\dim(A_i) < \dim(A)$ when $i < n$ and $\dim(A_i) = \dim(A)$ when $i \geq n$.
For each $n \leq i \leq m$ we let $I_i$ be the relative interior of $A_i$ in $A$.
As $A_n \cup \ldots \cup A_m$ is almost all of $A$ Corollary~\ref{corollary:int2} shows that almost every element of $A$ is an element of some $I_i$.
It suffices to fix $i \geq n$ and show that the proposition holds for some $p \in I_i$.
By Lemma~\ref{lem:localhomeo} there is an open $U \subseteq M^k$ and a coordinate projection $\pi: M^k \to M^d$ such that $p \in U$, $\pi(U \cap A_i)$ is open, and the restriction of $\pi$ to $U \cap A_i$ is a homeomorphism onto its image.
After replacing $U$ with a smaller open set if necessary we may assume that $U \cap A \subseteq A_i$.
We let $V = U \cap A$.
Then $\pi(V)$ is open and $\pi|_{V}$ is a homeomorphism onto its image.
\end{proof}

\begin{Prop}\label{prp:finitefibers}
Let $f: M^k \rightarrow M^l$ be a definable function such that $|f^{-1}(p)| < \infty$ for all $p \in M^l$.
Almost every $p \in M^k$ has a neighborhood $V$ such that the restriction of $f$ to $V$ is injective.
\end{Prop}

\begin{proof}
Let $\Delta$ be the diagonal $\{ (x,x) : x \in M^k\}$ in $M^k \times M^k$.
Let $D \subseteq M^k \times M^k$ be the set of $(x,y) \in M^k \times M^k$ such that $x \neq y$ and $f(x) = f(y)$.
For each $x \in M^k$ there are at most finitely many $y \in M^k$ such that $(x,y) \in D$.
Thus $\dim(D) \leqslant k$ and so $\dim \partial(D) < k$.
As $\Delta$ and $D$ are disjoint this implies that $\dim( \cl(D) \cap \Delta) < k$.
Let $B$ be the set of $p \in M^k$ such that $(p,p) \notin \cl(D)$.
Then $B$ is almost all of $M^k$.
Fix $p \in B$.
There is an open neighborhood $V$ of $p$ such that $[V \times V] \cap D  =\emptyset$.
If $x,y \in V$ and $x \neq y$ then as $(x,y) \notin D$ we have $f(x) \neq f(y)$.
Thus $f$ is injective on $V$.
\end{proof}

\section{One-variable Functions}
In this final section we prove two results about one-variable functions.
\begin{Prop}\label{prp:1varhomeo}
Let $f: M \rightarrow M$ be a definable function.
All but finitely many $p \in M$ have an open neighborhood $V$ on which one of the following holds:
\begin{enumerate}
\item the restriction of $f$ to $V$ is constant;
\item $f(V)$ is open and the restriction of $f$ to $V$ is a homeomorphism onto its image.
\end{enumerate}
\end{Prop}

\begin{proof}
It is enough to show that the set of $p$ satisfying either (1) or (2) above is dense.
Fix a definable open $U \subseteq M$.
We show that $U$ contains a point at which either (1) or (2) holds.
We first suppose that the restriction of $f$ to $U$ does not have finite fibers.
Then there is a $p \in U$ for which there are infinitely many $q \in U$ satisfying $f(q) = f(p)$.
This implies that there is a definable open $V \subseteq U$ such that $f(q) = f(p)$ for all $q \in V$.
Then $(1)$ holds at any point in $V$.
We now suppose that $f|_U$ has finite fibers.
After applying Proposition~\ref{proposition:continuity} and replacing $U$ with a smaller definable open set if necessary we suppose that $f$ is continuous on $U$.
After applying Proposition~\ref{prp:finitefibers} and replacing $U$ with a smaller definable open set if necessary we assume that $f|_U$ is injective.
Then $f(U)$ is infinite and thus contains a definable open set $W$.
By Proposition~\ref{proposition:continuity} there is a definable open $W' \subseteq W$ such that $(f|_U)^{-1}$ is continuous on $W'$.
Then $(f|_U)^{-1}(W')$ is infinite and thus contains a definable open set $V \subseteq U$.
The restriction of $f$ to this $V$ is a homeomorphism onto its image.
\end{proof}
\noindent Finally, we characterize when algebraic closure on $M$ admits exchange.
\begin{Prop}
Exactly one of the following holds:
\begin{enumerate}
\item there is a non-empty definable open $U \subseteq M$ and a locally constant correspondence $U \tto M$ with infinite image.
\item $\acl$ satisfies exchange.
\end{enumerate}
\end{Prop}

\begin{proof}
Let $\pi_1,\pi_2 : M^2 \rightarrow M$ be the projections onto the first and second coordinates, respectively.
We first suppose that $\acl$ satisfies exchange and show that (1) does not hold.
Suppose towards a contradiction that $U \subseteq M$ is definable and open and that $f: U \tto M$ is a locally constant correspondence with infinite image.
The restriction of $\pi_1$ to $\gr(f)$ has finite fibers hence:
$$\dim \gr(f) = \dim(U) = 1.$$
If $(a,b) \in \gr(f)$ then there is an open neighbourhood $V \subseteq U$ of $a$ such that $(a',b) \in \gr(f)$ for all $a' \in V$.
Therefore the restriction of $\pi_2$ to $\gr(f)$ has infinite fibers, so as $\acl$ admits exchange:
$$\dim \gr(f) \geq 1 + \dim \pi_2[\gr(f)].$$
As $f$ has infinite image $\dim \pi_2[\gr(f)] = 1$ so $\dim \gr(f) = 2$, contradiction.
We now suppose that $\acl$ does not satisfy exchange.
Then there is a set of parameters $K \subseteq M$ and $a,b \in M$ such that:
$$b \in \acl(K \cup \{a\}) \setminus \acl(K) \quad \text{and} \quad a \notin \acl(K \cup \{b\}).$$
This implies that there is a $K$-definable $D \subseteq M \times M$ such that $(a,b) \in D$ and for every $a' \in M$ and there are only finitely many $b' \in M$ such that $(a',b') \in D$.
As $a \notin \acl(K \cup \{b\} )$, $a$ is an interior point of $\{ x \in M : (x,b) \in D\}$.
Let $D'$ be the $K$-definable set of $(p,q) \in D$ such that $p$ is an interior point of $\{ x \in M : (x,q) \in D \}$.
If $(p,q) \in D'$ then $p$ is an interior point of $\{ x \in M : (x,q) \in D' \}$.
After replacing $D$ with $D'$ if necessary we suppose that $p$ is an interior point of $\{ x \in M : (x,q) \in D\}$ for all $(p,q) \in D$.
This implies that $\pi_1(D)$ is open.
We declare $V = \pi_1(D)$ and let $g: V \tto M$ be the $K$-definable correspondence such that $\gr(f) = D$.
If $q \in g(p)$ for some $p \in V$ then $p$ is in the interior of $\{ x \in M : q \in f(x) \}$.
Let $N$ be such that $|g(p)| \leqslant N$ for all $p \in V$.
For $1 \leqslant i \leqslant N$ let $E_i$ be the set of $p \in V$ such that $|g(p)| = i$.
As
$$ |V \setminus [\interior{E_1} \cup \ldots \cup \interior{E_N}]| < \infty $$
we have $a \in \interior{E_n}$ for some $n$.
We let $U = \interior{E_n}$ and $f$ be the restriction of $g$ to $U$.
As $b$ is in the image of $f$ and $b \notin \acl(K)$, $f$ must have infinite image.
We show that $f$ is locally constant.
Let $p \in U$ and $f(p) = \{q_1,\ldots,q_n\}$.
It follows by definition of $V$ that for every $1 \leqslant i \leqslant n$ we can choose a neighborhood $W_i \subseteq U$ of $p$ such that $q_i \in f(p')$ for any $p' \in W_i$.
Let $W$ be the intersection of the $W_i$.
If $p' \in W$ then $\{q_1,\ldots, q_n\} \subseteq f(p)$.
As $p' \in E_n$ we have $|f(p)| = n$ so $\{q_1,\ldots,q_n\} = f(p)$.
Thus $f(p)$ is constant on $W$. 
\end{proof}

\bibliographystyle{alpha}
\bibliography{franzi}
\end{document}